\documentclass[10pt]{amsart}
\usepackage{amsmath,amscd}
\usepackage{amsbsy}
\usepackage{amssymb}
\usepackage{amscd,amsthm}
\usepackage[all,cmtip]{xy}

\newtheorem{thm}{Theorem}
\numberwithin{thm}{section}
\newtheorem{lem}[thm]{Lemma}
\newtheorem{prop}[thm]{Proposition}
\newtheorem{cor}[thm]{Corollary}

\newtheorem{rema}[thm]{Remark}

\newtheorem{que}[thm]{Question}

\newtheorem{defi}[thm]{Definition}

\newtheorem*{thm2}{Theorem}
\newtheorem*{cor2}{Corollary}

\newtheorem*{que2}{Question}

\begin{document}
	\begin{center}
		\huge{Rational points on symmetric powers and categorical representability}\\[1cm]
	\end{center}
	
	\begin{center}
		
		\large{Sa$\mathrm{\check{s}}$a Novakovi$\mathrm{\acute{c}}$}\\[0,4cm]
		{\small December 2019}\\[0,3cm]
	\end{center}

	\noindent{\small \textbf{Abstract}. 
		In this paper we observe that for geometrically integral projective varieties $X$, admitting a full weak exceptional collection consisting of pure vector bundles, the existence of a $k$-rational point implies $\mathrm{rdim}(X)=0$. We also study the symmetric power $S^n(X)$ of Brauer--Severi and involution varieties over $\mathbb{R}$ and prove that the equivariant derived category $D^b_{S_n}(X^n)$ admits a full weak exceptional collection. As a consequence, we find $\mathrm{rdim}(X)=0$ if and only if $\mathrm{rdim}(D^b_{S_n}(X^n))=0$ for $1\leq n\leq 3$. If $X$ is Brauer--Severi, the existence of a $\mathbb{R}$-rational point on $X$ or $S^3(X)$ is equivalent to $\mathrm{rdim}(D^b_{S_3}(X^3))=0$. \\
		
		\section{Introduction}
When the base field $k$ is not algebraically closed, the existence of $k$-rational points on $X$ (being a necessary condition for rationality) is a major open problem in arithmetic geometry. In \cite{ABT} Auel and Bernardara formulated the following question, actually posed by H. Esnault in 2009:
\begin{que}
	Let $X$ be a smooth projective variety over a field $k$. Can the bounded derived category $D^b(X)$ detect the existence of a $k$-rational point on $X$?
\end{que}
This question is now central for arithmetic aspects of the theory of derived categories, see \cite{ABT}, \cite{ANT}, \cite{ANTT}, \cite{ASCHT}, \cite{ABBT}, \cite{ABB1T}, \cite{HTT}, \cite{LOT} \cite{NO2T} and \cite{NO3T}.
In the present work we use a potential measure for rationality to answer Question 1.1 affirmatively in some special cases. In \cite{BBT} Bernardara and Bolognesi introduced the notion of categorical representability. Below we use the definition given in \cite{AB1T}. So a $k$-linear triangulated category $\mathcal{T}$ is said to be \emph{representable in dimension $m$} if there is a semiorthogonal decomposition $\mathcal{T}=\langle \mathcal{A}_1,...,\mathcal{A}_n\rangle$ and for each $i=1,...,n$ there exists a smooth projective connected variety $Y_i$ with $\mathrm{dim}(Y_i)\leq m$, such that $\mathcal{A}_i$ is equivalent to an admissible subcategory of $D^b(Y_i)$ (see \cite{AB1T}). We use the following notation
\begin{eqnarray*}
	\mathrm{rdim}(\mathcal{T}):=\mathrm{min}\{m\mid \mathcal{T}\  \textnormal{is representable in dimension m}\},
\end{eqnarray*}
whenever such a finite $m$ exists. Let $X$ be a smooth projective $k$-variety. One says $X$ is \emph{representable in dimension} $m$ if $D^b(X)$ is representable in dimension $m$. We will use the following notation:
\begin{eqnarray*}
	\mathrm{rdim}(X):=\mathrm{rdim}(D^b(X)).
\end{eqnarray*}
Quite recently, it has been shown that certain varieties $X$ admit $k$-rational points if and only if $\mathrm{rdim}(X)=0$, see \cite{ABT}, \cite{NO2T} and \cite{NO3T}. Among these varieties are certain Fano varieties having full weak exceptional collections. For arbitrary varieties admitting full weak exceptional collections, we observe the following:
\begin{thm2}[Theorem 6.3]
	Let $X$ be a smooth projective variety over a field $k$ satisfying $H^0(X^s,\mathcal{O}_{X^s})=k^s$ and assume $D^b(X)$ admits a full weak exceptional collection consisting of pure vector bundles. If $X(k)\neq \emptyset$, then $\mathrm{rdim}(X)=0$.
\end{thm2}
Recall that a Fano variety $X$ of dimension 1 is a Brauer--Severi curve. In this case $D^b(X)$ has a full weak exceptional collection consisting of pure vector bundles. Results in \cite{NO2T} show that for finite products $Y$ of Brauer--Severi varieties, $Y(k)\neq \emptyset$ if and only if $\mathrm{rdim}(Y)=0$. Del Pezzo surfaces $S$ admitting a full weak exceptional collection consisting of pure vector bundles or, more generally, admitting a semiorthogonal decomposition of the form
\begin{eqnarray*}
	D^b(S)=\langle D^b(l_1/k,A_1),...,D^b(l_n/k,A_n)\rangle,
\end{eqnarray*}
where $l_i/k$ are field extensions and $A_i$ suitable central simple algebras over $l_i$ must be of degree $d\geq 5$ (see \cite{ABT}). In any of these cases, Auel and Bernardara \cite{ABT} proved $S(k)\neq \emptyset$ if and only if $\mathrm{rdim}(S)=0$. Note that there are also Fano threefolds $X$, such as Brauer--Severi varieties or twisted forms of quadrics which admit full weak exceptional collections, respectively semiorthogonal decompositions of the form
\begin{eqnarray*}
	D^b(X)=\langle D^b(l_1/k,A_1),...,D^b(l_n/k,A_n)\rangle.
\end{eqnarray*}
The results in \cite{NO2T} show that $X(k)\neq \emptyset$ if and only if $\mathrm{rdim}(X)=0$. On the other hand, with the help of \cite{BTT}, Proposition 1.7 one can cook up examples of anisotropical quadrics admitting full exceptional collections which are not rational and have no $k$-rational points. So there are examples of Fano varieties $X$ for which $\mathrm{rdim}(X)=0$ does not imply the existence of a $k$-rational point. Motivated by these examples, we want to ask the following question.
\begin{que}
	Let $X$ be a smooth Fano variety over a field $k$. Suppose $D^b(X)$ admits a full weak exceptional collection consisting of pure vector bundle, or more generally a semiorthogonal decomposition of the form
	\begin{eqnarray*}
		D^b(X)=\langle D^b(l_1/k,A_1),...,D^b(l_n/k,A_n)\rangle,
	\end{eqnarray*}
	where $l_i/k$ are field extensions and $A_i$ suitable central simple algebras over $l_i$. Is there any characterization for which $X$, $\mathrm{rdim}(X)=0$ implies $X(k)\neq \emptyset$?
\end{que}
The main goal of the present paper is to modify Question 1.1 by studying singular projective varieties. However, one has to seek for the "right" analogue of the derived category $D^b(X)$. In Section 6 we study symmetric powers of Brauer--Severi varieties $X$ over $\mathbb{R}$ and its associated equivariant derived category $D^b_{S_n}(X^n)$. This consideration is motivated by a paper of Krashen and Saltman \cite{KST} in which they studied the question whether for Brauer--Severi varieties $X$ the rationality of the symmetric power $S^n(X)$ implies rationality of $X$. In this context we also want to mention a paper of Koll\'ar \cite{KOLT} where products of symmetric powers of a Brauer--Severi variety are classified up-to stable birational equivalence. In the present work we want to shed light on the existence of rational points on $S^n(X)$ from a derived point of view by using the concept of categorical representability. Notice that in some cases the existence of a rational point on $S^n(X)$ forces $X$ to be rational and we can relate $\mathrm{rdim}(X)$ to $\mathrm{rdim}(D^b_{S_n}(X^n))$. We first show the following result which is not a surprise and follows from available results and techniques in the literature. We focus on Brauer--Severi varieties and involution vartieies of orthogonal type over $\mathbb{R}$ since we make use of the endomorphism algebras of the irreducible representations of the symmetric group $S_n$. At some point in the proof we need the endomorphism algebras to be central simple or matrix algebras over a field. This is quite hard to accomplish for Brauer--Severi and involution varieties over arbitrary fields $k$. 
\begin{thm2}[Theorem 6.5]
	Let $X$ be a Brauer--Severi variety over $\mathbb{R}$ or a twisted quadric associated to a central simple $\mathbb{R}$-algebra $(A,\sigma)$ with involution of orthogonal type. Then $D^b_{S_n}(X^n)$ admits a full weak exceptional collection.
\end{thm2}
\noindent
With the help of Theorem 6.5, we can prove the following:
\begin{thm2}[Theorem 6.7]
	Let $X$ be a Brauer--Severi variety over $\mathbb{R}$ or a twisted quadric associated to a central simple $\mathbb{R}$-algebra $(A,\sigma)$ with involution of orthogonal type and $1\leq n\leq 3$. We set $\mathcal{T}:=D^b_{S_n}(X^n)$. Then the following hold:
	\begin{itemize}
		\item[(\textbf{i})] $\mathrm{rdim}(X)=0$ if and only if $\mathrm{rdim}(\mathcal{T})=0$,
		\item[(\textbf{ii})] $X(\mathbb{R})\neq \emptyset$ if and only if $\mathrm{rdim}(\mathcal{T})=0$.
	\end{itemize}
\end{thm2}
The proof of Theorem 6.7 in particular shows that the implication $\mathrm{rdim}(\mathcal{T})=0 \Rightarrow \mathrm{rdim}(X)=0$ holds for arbitrary positive integers $n$. As a consequence of the latter result we find:
\begin{cor}[Corollary 6.9]
	Let $X$ be a Brauer--Severi variety over $\mathbb{R}$ corresponding to $A$ and assume $\mathrm{deg}(A)>3$. Then $S^3(X)(\mathbb{R})\neq \emptyset$ if and only if $\mathrm{rdim}(\mathcal{T})=0$.
\end{cor}  
The existence of rational points seems, in general, not to be related to categorical representability in dimension zero. For instance, an elliptic curve over a number field is categorical representable in dimension one (see \cite{AB1T}) although it has rational points. Indeed, the rationality of a given variety $X$ seems to be related to categorical representability in codimension 2. The definition of categorical representability from above works perfectly for smooth varieties. For singular varieties however, one has to involve so called noncommutative resolutions of singularities. We recall the following definition from \cite{MBTT}: let $X$ be a projective, possibly singular, variety over $k$ and a $\mathcal{A}$ a noncommutative resolution of singularities in the sense of \cite{MBTT} (see Section 4 for the definition). $X$ is said to be \emph{representable in dimension $m$} if there is a semiorthogonal decomposition $\mathcal{A}=\langle \mathcal{A}_1,...,\mathcal{A}_n\rangle$ and for each $i=1,...,n$ there exists a smooth projective connected variety $Y_i$ with $\mathrm{dim}(Y_i)\leq m$, such that $\mathcal{A}_i$ is equivalent to an admissible subcategory of $D^b(Y_i)$ (see \cite{MBTT}). We use the following notations
\begin{center}
	\begin{tabular}{c c}
		$\mathrm{rdim}(X):=\mathrm{min}\{m\mid X \  \textnormal{is representable in dimension m}\}$,\\
		\\
		$\mathrm{rcodim}(X):=\mathrm{dim}(X) - \mathrm{rdim}(X)$.
	\end{tabular}
\end{center}
\begin{que2}[\cite{MBTT}]
	Let $X$ be a projective variety over $k$ of dimension at least 2. Suppose $X$ is $k$-rational. Do we have $\mathrm{rcodim}(X)\geq 2$ ?
\end{que2}
\noindent
Assuming $\mathrm{dim}(X)\geq 2$, Theorem 6.3 in particular implies that if $X$ is $k$-rational, then $\mathrm{rcodim}(X)\geq 2$. Moreover, in view of the latter question, we show that the dg enhancement of $D^b_{S_n}(X^n)$ is a noncommutative resolution of singularities of the noncommutative scheme $\mathrm{perf}(S^3(X))$ (see Theorem 6.14). Then Corollary 6.9 from above says: 
\begin{cor2}[Corollary 6.15]
	Let $X$ be a Brauer--Severi variety over $\mathbb{R}$ corresponding to $A$ and assume $\mathrm{deg}(A)>3$. If $S^3(X)$ is $\mathbb{R}$-rational, we have $\mathrm{rcodim}(S^3(X))\geq 2$.
\end{cor2}

{\small \textbf{Acknowledgement}. I wish to thank Zinovy Reichstein for answering questions and providing me with literature for group actions on central simple algebras. I also like to thank Asher Auel, Marcello Bernardara and Michele Bolognesi, whose articles inspired me to deal with this subject. Finally, I would like to thank the Heinrich--Heine--University for financial support via the SFF-grant.}\\

{\small \textbf{Notations}. If $X$ is a $k$-variety, we will denote by $D^b(X)$ the bounded derived category of complexes of coherent sheaves on $X$. Notice that $D^b(X)$ is a $k$-linear category. Let $B$ be an $\mathcal{O}_X$-algebra, we will denote by $D^b(X,B)$ the bounded derived category of complexes of $B$-modules, considered as a $k$-linear category. For $X=\mathrm{Spec}(K)$ and $B$ associated to a $K$-algebra, we will write $D^b(K/k,B)$. Also $D^b(K,B)$ is shorthanded for $D^b(K/K,B)$.} 

\section{Brauer--Severi and involution varieties}
A \emph{Brauer--Severi variety} of dimension $n$ is a $k$-variety $X$ such that $X\otimes_k L\simeq \mathbb{P}^n$ for a finite field extension $L/k$. An extension $L/k$ for which $X\otimes_k L\simeq \mathbb{P}^n$ is called \emph{splitting field} of $X$. In fact, every Severi--Brauer variety always splits over a finite separable field extension of $k$ (see \cite{GST}, Corollary 5.1.4) and therefore over a finite Galois extension. It follows from descent theory that $X$ is projective, integral and smooth over $k$. Via Galois cohomology, isomorphism classes of $n$-dimensional Severi--Brauer varieties are in one-to-one correspondence with isomorphism classes of central simple $k$-algebras of degree $n+1$. If $A$ is a central simple algebra, we will write $BS(A)$ for the corresponding Brauer--Severi variety.

Recall, a finite-dimensional $k$-algebra $A$ is called \emph{central simple} if it is an associative $k$-algebra that has no two-sided ideals other than $0$ and $A$ and if its center equals $k$. Note that $A$ is a central simple if and only if there is a finite field extension $L/k$, such that $A\otimes_k L \simeq M_n(L)$ (see \cite{GST}, Theorem 2.2.1). 
An extension $L/k$ such that $A\otimes_k L\simeq M_n(L)$ is called \emph{splitting field} for $A$. The \emph{degree} of a central simple algebra $A$ is defined to be $\mathrm{deg}(A):=\sqrt{\mathrm{dim}_k A}$. According to the \emph{Wedderburn Theorem} for any central simple $k$-algebra $A$ there is an integer $n>0$ and a division algebra $D$, such that $A\simeq M_n(D)$. The division algebra $D$ is also central and unique up to isomorphism. The degree of the unique central division algebra $D$ is called the \emph{index} of $A$ and is denoted by $\mathrm{ind}(A)$. 

Two central simple $k$-algebras $A\simeq M_n(D)$ and $B\simeq M_m(D')$ are called \emph{Brauer-equivalent} if $D\simeq D'$. The \emph{Brauer group} $\mathrm{Br}(k)$ of a field $k$ is the group whose elements are equivalence classes of central simple $k$-algebras, with addition given by the tensor product of algebras. It is an abelian group with inverse of the equivalence class $[A]$ being $[A^{op}]$. The neutral element is $[k]$. It is a fact that the Brauer group of any field is a torsion group. The order of an equivalence class $[A]\in \mathrm{Br}(k)$ is called the \emph{period} of $[A]$ and is denoted by $\mathrm{per}(A)$. For instance if $k=\mathbb{R}$, the Brauer group $\mathrm{Br}(\mathbb{R})$ is cyclic of order two and generated by the equivalence class of the Hamiltion quaternions $\mathbb{H}$.

To a central simple algebra $A$ of degree $n$ with involution $\sigma$ of the first kind over a field $k$ of $\mathrm{char}(k)\neq 2$ one can associate the \emph{involution variety} $\mathrm{IV}(A,\sigma)$. This variety can be described as the variety of $n$-dimensional right ideals $I$ of $A$ such that $\sigma(I)\cdot I=0$. If $A$ is split so $(A,\sigma)\simeq (M_n(k), q^*)$, where $q^*$ is the adjoint involution defined by a quadratic form $q$ one has $\mathrm{IV}(A,\sigma)\simeq V(q)\subset \mathbb{P}^{n-1}_k$. Here $V(q)$ is the quadric determined by $q$. By construction such an involution variety  $\mathrm{IV}(A,\sigma)$ becomes a quadric in $\mathbb{P}^{n-1}_L$ after base change to some splitting field $L$ of $A$. In this way the involution variety is a twisted form of a smooth quadric. Recall from \cite{DTT} that a splitting field $L$ splits $A$ \emph{isotropically} if $(A,\sigma)\otimes_k L\simeq (M_n(L), q^*)$ with $q$ an isotropic quadratic form over $L$. Although the degree of $A$ is arbitrary, (when $\mathrm{char}(k)\neq 2$), the case where degree of $A$ is odd does not give anything new, since central simple algebras of odd degree with involution of the first kind are split (see \cite{KMRT}, Corollary 2.8). For details on the construction and further properties on involution varieties and the corresponding algebras we refer to \cite{DTT}.


\section{Semiorthogonal decomposition and exceptional collections}
Let $\mathcal{D}$ be a triangulated category and $\mathcal{A}$ a triangulated subcategory. The subcategory $\mathcal{A}$ is called \emph{thick} if it is closed under isomorphisms and direct summands. For a subset $M$ of objects of $\mathcal{D}$ we denote by $\langle M\rangle$ the smallest full thick subcategory of $\mathcal{D}$ containing the elements of $M$.
Furthermore, we define $M^{\perp}$ to be the subcategory of $\mathcal{D}$ consisting of all objects $C$ such that $\mathrm{Hom}_{\mathcal{D}}(E[i],C)=0$ for all $i\in \mathbb{Z}$ and all elements $E$ of $M$. We say that $M$ \emph{generates} $\mathcal{D}$ if $M^{\perp}=0$. Now assume $\mathcal{D}$ admits arbitrary direct sums. An object $F$ is called \emph{compact} if $\mathrm{Hom}_{\mathcal{D}}(F,-)$ commutes with direct sums. Denoting by $\mathcal{D}^c$ the subcategory of compact objects we say that $\mathcal{D}$ is \emph{compactly generated} if the objects of $\mathcal{D}^c$ generate $\mathcal{D}$. Let $\mathcal{D}$ be a compactly generated triangulated category. Then a set of objects $A\subset \mathcal{D}^c$ generates $\mathcal{D}$ if and only if $\langle A\rangle=\mathcal{D}^c$ (see \cite{BVT}, Theorem 2.1.2).

Let $G$ be a finite group acting on a smooth projective variety $X$ over a field $k$ and assume $\mathrm{char}(k)\nmid\mathrm{ord}(G)$. The \emph{equivariant derived category}, denoted by $D^b_G(X)$, is defined to be $D^b(\mathrm{Coh}_G(X))$. For details see for instance \cite{NO1T}, Section 2. For any two objects $\mathcal{F}^{\bullet}, \mathcal{G}^{\bullet}\in D^b_G(X)$ we write $\mathrm{Hom}_G(\mathcal{F}^{\bullet},\mathcal{G}^{\bullet}):=\mathrm{Hom}_{D^b_G(X)}(\mathcal{F}^{\bullet},\mathcal{G}^{\bullet})$. As $\mathrm{char}(k)\nmid\mathrm{ord}(G)$, the functor $(-)^G$ is exact and one has
\begin{center}
	$\mathrm{Hom}_{G}(\mathcal{F}^{\bullet},\mathcal{G}^{\bullet}[i])\simeq \mathrm{Hom}(\mathcal{F}^{\bullet},\mathcal{G}^{\bullet}[i])^G$.
\end{center}
Recall, that for every subgroup $H\subset G$, the restriction functor $\mathrm{Res}\colon D^b_G(X)\rightarrow D^b_H(X)$ has the inflation functor $\mathrm{Inf}\colon D^b_H(X)\rightarrow D^b_G(X)$ as a left and right adjoint, see \cite{ET}, Section 3. It is given for $\mathcal{A}\in D^b_H(X)$ by
\begin{eqnarray*}
	\mathrm{Inf}^G_H(\mathcal{A})=\bigoplus_{[g]\in H\setminus G} g^*\mathcal{A}.
\end{eqnarray*}

\begin{defi}
	\textnormal{Let $A$ be a division algebra over $k$, not necessarily central. An object $\mathcal{E}^{\bullet}\in D^b_G(X)$ is called \emph{weak exceptional} if $\mathrm{End}_G(\mathcal{E}^{\bullet})=A$ and $\mathrm{Hom}_G(\mathcal{E}^{\bullet},\mathcal{E}^{\bullet}[r])=0$ for $r\neq 0$. If $A=k$ the object is called \emph{exceptional}.}
\end{defi}
\begin{defi}
	\textnormal{A totally ordered set $\{\mathcal{E}^{\bullet}_1,...,\mathcal{E}^{\bullet}_n\}$ of weak exceptional objects in $D^b_G(X)$ is called an \emph{weak exceptional collection} if $\mathrm{Hom}_G(\mathcal{E}^{\bullet}_i,\mathcal{E}^{\bullet}_j[r])=0$ for all integers $r$ whenever $i>j$. An weak exceptional collection is \emph{full} if $\langle\{\mathcal{E}^{\bullet}_1,...,\mathcal{E}^{\bullet}_n\}\rangle=D^b_G(X)$ and \emph{strong} if $\mathrm{Hom}_G(\mathcal{E}^{\bullet}_i,\mathcal{E}^{\bullet}_j[r])=0$ whenever $r\neq 0$. If the set $\{\mathcal{E}^{\bullet}_1,...,\mathcal{E}^{\bullet}_n\}$ consists of exceptional objects it is called \emph{exceptional collection}}. 
\end{defi}
A generalization of the notion of a full weak exceptional collection is that of a semiorthogonal decomposition of $D^b_G(X)$. Recall that a full triangulated subcategory $\mathcal{A}$ of $D^b_G(X)$ is called \emph{admissible} if the inclusion $\mathcal{D}\hookrightarrow D^b_G(X)$ has a left and right adjoint functor.
\begin{defi}
	\textnormal{Let $X$ be a smooth projective variety over $k$. A sequence $\mathcal{A}_1,...,\mathcal{A}_n$ of full triangulated subcategories of $D^b_G(X)$ is called \emph{semiorthogonal} if all $\mathcal{A}_i\subset D^b_G(X)$ are admissible and $\mathcal{A}_j\subset \mathcal{A}_i^{\perp}=\{\mathcal{F}^{\bullet}\in D^b_G(X)\mid \mathrm{Hom}_G(\mathcal{G}^{\bullet},\mathcal{F}^{\bullet})=0$, $\forall$ $ \mathcal{G}^{\bullet}\in\mathcal{A}_i\}$ for $i>j$. Such a sequence defines a \emph{semiorthogonal decomposition} of $D^b_G(X)$ if the smallest full thick subcategory containing all $\mathcal{A}_i$ equals $D^b_G(X)$.}
\end{defi}
\noindent
For a semiorthogonal decomposition we write $D^b_G(X)=\langle \mathcal{A}_1,...,\mathcal{A}_n\rangle$.

\section{Descent for vector bundles}
Let $X$ be a proper $k$-variety and $\mathcal{W}$ an indecomposable vector bundle on $X\otimes_ k k^s$. A vector bundle $\mathcal{V}$ on $X$ is called \emph{pure of type $\mathcal{W}$} if $\mathcal{V}\otimes_k k^s\simeq \mathcal{W}^{\oplus m}$ (see \cite{AET}). We say $\mathcal{V}$ is \emph{pure} if it is pure of type $\mathcal{L}$ for a line bundle $\mathcal{L}$. Note that $\mathcal{L}$ is $G_s:=\mathrm{Gal}(k^s|k)$ invariant. Recall the Brauer obstruction for  invariant line bundles on smooth proper geometrically integral $k$-varieties $X$. The sequence of low degree terms of the Leray spectral sequence is  
\begin{eqnarray*}
	0\rightarrow \mathrm{Pic}(X)\rightarrow \mathrm{Pic}(X_{k^s})^{G_s}\stackrel{d}{\rightarrow} \mathrm{Br}(X)\rightarrow \mathrm{Br}(X).
\end{eqnarray*}
It is well-known that for proper varieties over $k$ one has $\mathrm{Pic}(X_{k^s})^{G_s}\simeq \mathrm{Pic}_{(X/k)(et)}(k)$. Recall from \cite{LIT} that a field extension $L/k$ is called \emph{splitting field} for $\mathcal{L}\in \mathrm{Pic}_{(X/k)(et)}(k)$ if $\mathcal{L}\otimes_k L$ lies $\mathrm{Pic}(X\otimes_k L)$. Moreover, the class $\mathcal{L}$ is called \emph{globally generated} if there is a splitting field $K$ of $\mathcal{L}$ such that $\mathcal{L}\otimes_k L$ is globally generated. The following result will be applied in Section 6.
\begin{thm}[\cite{LIT}, Theorem 1.1]
	Let $X$ be a proper variety over a field $k$ and $P$ a Brauer--Severi variety. Then there exists a morphism $\phi\colon X\rightarrow P$ if and only if there is a globally generated $\mathcal{L}\in \mathrm{Pic}_{(X/k)(et)}(k)$ with $d(\mathcal{L})=[P]\in \mathrm{Br}(k)$.
\end{thm}

\section{Noncommutative motives of central simple and separable algebras}
We refer to the book \cite{GTAT} (alternatively see \cite{TTT} and \cite{MTT} for a survey on noncommutative motives).
Let $\mathcal{A}$ be a small dg category. The homotopy category $H^0(\mathcal{A})$ has the same objects as $\mathcal{A}$ and as morphisms $H^0(\mathrm{Hom}_{\mathcal{A}}(x,y))$. A source of examples is provided by schemes since the derived category of perfect complexes $\mathrm{perf}(X)$ of any quasi-compact quasi-separated scheme $X$ admits a canonical dg enhancement $\mathrm{perf}_{dg}(X)$ (for details see \cite{KELT}). Denote by $\textbf{dgcat}$ the category of small dg categories. The \emph{opposite} dg category $\mathcal{A}^{op}$ has the same objects as $\mathcal{A}$ and $\mathrm{Hom}_{\mathcal{A}^{op}}(x,y):=\mathrm{Hom}_{\mathcal{A}}(y,x)$. A \emph{right $\mathcal{A}$-module} is a dg functor $\mathcal{A}^{op}\rightarrow C_{dg}(k)$ with values in the dg category $C_{dg}(k)$ of complexes of $k$-vector spaces. We write $C(\mathcal{A})$ for the category of right $\mathcal{A}$-modules. Recall form \cite{KELT} that the \emph{derived category} $D(\mathcal{A})$ of $\mathcal{A}$ is the localization of $C(\mathcal{A})$  with respect to quasi-isomorphisms. A dg functor $F\colon \mathcal{A}\rightarrow \mathcal{B}$ is called \emph{derived Morita equivalence} if the restriction of scalars functor $D(\mathcal{B})\rightarrow D(\mathcal{A})$ is an equivalence. The \emph{tensor product} $\mathcal{A}\otimes \mathcal{B}$ of two dg categories is defined as follows: the set of objects is the cartesian product of the sets of objects in $\mathcal{A}$ and $\mathcal{B}$ and $\mathrm{Hom}_{\mathcal{A}\otimes \mathcal{B}}((x,w),(y,z)):=\mathrm{Hom}_{\mathcal{A}}(x,y)\otimes\mathrm{Hom}_{\mathcal{B}}(w,z)$ (see \cite{KELT}). Given two dg categories $\mathcal{A}$ and $\mathcal{B}$, let $\mathrm{rep}(\mathcal{A},\mathcal{B})$ be the full triangulated subcategory of $D(\mathcal{A}^{op}\otimes \mathcal{B})$ consisting of those $\mathcal{A}\text{-}\mathcal{B}$-bimodules $M$ such that $M(x,-)$ is a compact object of $D(\mathcal{B})$ for every object $x\in \mathcal{A}$. 
The category $\textbf{dgcat}$ of all (small) dg categories and dg functors carries a Quillen model structure whose weak equivalences are Morita equivalences. Let us denote by $\mathrm{Hmo}$ the homotopy category hence obtained and by $\mathrm{Hmo}_0$ its additivization. Now to any small dg category $\mathcal{A}$ one can associate functorially its \emph{noncommutative motive} $U(\mathcal{A})$ which takes values in $\mathrm{Hmo}_0$. This functor $U\colon \textbf{dgcat}\rightarrow \mathrm{Hmo}_0$ is proved to be the \emph{universal additive invariant} (see \cite{TA1T}). Recall that an additive invariant is any functor $E\colon \textbf{dgcat}$ $\rightarrow \mathcal{D}$ taking values in an additive category $\mathcal{D}$ such that
\begin{itemize}
	\item[(\textbf{i})] it sends derived Morita equivalences to isomorphisms,\\
	
	\item[(\textbf{ii})] for any pre-triangulated dg category $\mathcal{A}$ admitting full pre-triangulated dg subcategories $\mathcal{B}$ and $\mathcal{C}$ such that $H^0(\mathcal{A})=\langle H^0(\mathcal{B}), H^0(\mathcal{C})\rangle$ is a semiorthogonal decomposition, the morphism $E(\mathcal{B})\oplus E(\mathcal{C})\rightarrow E(\mathcal{A})$ induced by the inclusions is an isomorphism.
\end{itemize}			
For central simple $k$-algebras one has the following comparison theorem.
\begin{thm}[\cite{TAT}, Theorem 2.19]
	Let $A_1,...,A_n$ and $B_1,...,B_m$ be central simple $k$-algebras, then the following are equivalent:
	\begin{itemize}
		\item[(\textbf {i})] There is an isomorphism
		\begin{eqnarray*}
			\bigoplus^n_{i=1}U(A_i)\simeq \bigoplus^m_{j=1}U(B_j).
		\end{eqnarray*}
		
		\item[(\textbf {ii})] The equality $n=m$ holds and for all $1\leq i\leq n$ and all $p$
		\begin{eqnarray*}
			[B^p_i]=[A^p_{\sigma_p(i)}]\in \mathrm{Br}(k)
		\end{eqnarray*}
		for some permutations $\sigma_p$ depending on $p$.
	\end{itemize}
\end{thm}	
Later, we also need the following result.
\begin{prop}[\cite{TABT}, Proposition 4.5]
	Let $A_1,...,A_n$ and $B_1,...,B_m$ be central simple $k$-algebras, and $NM$ a noncommutative motive. If
	\begin{eqnarray*}
		\bigoplus^n_{i=1}U(A_i)\oplus NM\simeq \bigoplus^m_{j=1}U(B_j)\oplus NM,
	\end{eqnarray*}
	then $n=m$ and $U(A_1)\oplus\cdots \oplus U(A_n)\simeq U(B_1)\oplus\cdots \oplus U(B_n)$.
\end{prop}
\noindent
Recall the following definition (see for instance \cite{MBTT}).
\begin{defi}
	\textnormal{Let $\mathcal{T}$ be a triangulated category. A \emph{dg-enhancement} of $\mathcal{T}$ is a pair $(\mathcal{A}, \epsilon)$, where $\mathcal{A}$ is a pretriangulated dg category and $\epsilon\colon H^0(\mathcal{A})\rightarrow \mathcal{T}$ an equivalence of triangulated categories. Assume $\mathcal{T}$ has a dg-enhancement. Then $\mathcal{T}$ admits a \emph{unique} dg-enhancement if for any two dg-enhancements $(\mathcal{A},\epsilon)$ and $(\mathcal{A}', \epsilon')$ there is a dg-functor $\Psi\colon \mathcal{A}\rightarrow \mathcal{A}'$, inducing an equivalence $H^0(\Psi)\colon H^0(\mathcal{A})\rightarrow H^0(\mathcal{A}')$.}
\end{defi}



\section{Proofs of the results}
\begin{lem}
	Let $X$ be a proper variety over a field $k$ with $H^0(X^s,\mathcal{O}_{X^s})=k^s$. If $\mathcal{V}$ is a pure vector bundle on $X$, then $\mathrm{End}(\mathcal{V})$ is a central simple $k$-algebra.
\end{lem}
\begin{proof}
	As $\mathcal{V}$ is pure, there is a line bundle $\mathcal{L}\in \mathrm{Pic}(X\otimes_k k^s)$ such that $\mathcal{V}\otimes_k k^s\simeq \mathcal{L}^{\oplus m}$. The assertion then follows from the following isomorphisms
	\begin{eqnarray*}
		\mathrm{End}(\mathcal{V})\otimes_k k^s  \simeq  \mathrm{End}(\mathcal{L}^{\oplus m})\simeq \mathrm{End}(\mathcal{O}_{X^s}^{\oplus m})\simeq  M_{m}(k^s).
	\end{eqnarray*}
\end{proof}

\begin{prop}
	Let $X$ be a projective variety over a field $k$ with $H^0(X^s,\mathcal{O}_{X^s})=k^s$. Suppose $\mathcal{V}$ is a pure vector bundle on $X$. Then there is a morphism $X\rightarrow \mathrm{BS}(\mathrm{End}(\mathcal{V}))$.
\end{prop}
\begin{proof}
	We have $\mathcal{V}\otimes_k k^s\simeq \mathcal{L}^{\oplus m}$ for some line bundle $\mathcal{L}\in \mathrm{Pic}(X\otimes_k k^s)$ and Lemma 6.1 shows that $\mathrm{End}(\mathcal{V})$ is isomorphic to a central simple $k$-algebra $A$. 
	Let $E/k$ be a finite Galois extension within $k^s$ over which $\mathcal{V}$ is defined, i.e over which there exists a line bundle $\mathcal{N}$ such that $\mathcal{L}\simeq \mathcal{N}\otimes_E k^s$. Then \cite{WT}, Lemma 2.3 implies $\mathcal{V}\otimes_k E\simeq \mathcal{N}^{\oplus m}$. So we restrict to $X\otimes_k E$. On $X\otimes_k E$ we can find an ample line bundle $\mathcal{M}$. Then $\mathcal{M}'=\bigotimes_{g\in G}g^*\mathcal{M}$ is a $G:=\mathrm{Gal}(E|k)$-equivariant ample line bundle on $X\otimes_k E$. Note that for a suitable $n>0$, the line bundle $\mathcal{L}':=\mathcal{N}\otimes \mathcal{M}'^{\otimes n}$ is globally generated. From Section 4 we know $\mathcal{L}'\in \mathrm{Pic}_{(X/k)(et)}(k)$. Since $\mathcal{M}'^{\otimes n}$ descents to a line bundle $\mathcal{R}$ on $X$, we conclude 
	$(\mathcal{V}\otimes \mathcal{R})\otimes_k E\simeq \mathcal{L}'^{\oplus m}$. Now \cite{WT}, Lemma 2.3 shows that if $\mathcal{W}$ is another vector bundle satisfying $\mathcal{W}\otimes_k E\simeq \mathcal{L}'^{\oplus m}$, then $\mathcal{V}\otimes \mathcal{R}\simeq \mathcal{W}$. Hence
	$\mathrm{End}(\mathcal{W})\simeq \mathrm{End}(\mathcal{V})\simeq A$. Theorem 4.1 provides us with a morphism $X\rightarrow \mathrm{BS}(A)$.   
\end{proof}
\begin{thm}
	Let $X$ be a smooth projective variety over a field $k$ with $H^0(X^s,\mathcal{O}_{X^s})=k^s$ and assume $D^b(X)$ admits a full weak exceptional collection of pure vector bundles. If $X(k)\neq \emptyset$, then $\mathrm{rdim}(X)=0$. In particular, assuming $\mathrm{dim}(X)\geq 2$, if $X$ is $k$-rational, then $\mathrm{rcodim}(X)\geq 2$.
\end{thm}
\begin{proof}
	Denote by $\mathcal{V}_1,...,\mathcal{V}_r$ the full weak exceptional collection for $D^b(X)$. As all $\mathcal{V}_i$ are pure, Proposition 6.2 provides us with morphisms $X\rightarrow \mathrm{BS}(A_i)$, where $A_i=\mathrm{End}(\mathcal{V}_i)$. If $X$ admits a $k$-rational point, the Lang--Nishimura Theorem implies the existence of $k$-rational points on $\mathrm{BS}(A_i)$. Therefore all $\mathrm{BS}(A_i)$ are split. From this we conclude that $\mathcal{V}_i=\mathcal{L}_i^{\oplus m_i}$ for some line bundles $\mathcal{L}_i\in \mathrm{Pic}(X)$. But then $\{\mathcal{L}_1,...,\mathcal{L}_r\}$ forms a full exceptional collection in $D^b(X)$. Finally, Proposition 6.1.6 in \cite{AB1T} yields $\mathrm{rdim}(X)=0$.
\end{proof}

\begin{prop}
	Let $X$ be a Brauer--Severi variety over a field $k$ of characteristic zero corresponding to $A$ and $S^l(X)$ its symmetric power. Denote by $X_l$ the generalized Brauer--Severi variety associated to $A$ and let $l<\mathrm{deg}(A)$ be arbitrary. If $\mathrm{rdim}(X_l)=0$, then $S^l(X)(k)\neq \emptyset$.
\end{prop}
\begin{proof}
	According to \cite{KST}, Theorem 1.5 there is a birational map $X_l\times \mathbb{P}^{l(l-1)}\dashrightarrow S^l(X)$. If $\mathrm{rdim}(X_l)=0$, then \cite{NO2T}, Theorem 6.5 implies $X_l$ is a Grassmannian. Therefore $X_l(k)\neq \emptyset$. From the Lang--Nishimura Theorem we conclude $S^l(X)(k)\neq \emptyset$.
\end{proof}
\begin{thm}
	Let $X$ be a Brauer--Severi variety over $\mathbb{R}$ or a twisted quadric associated to a central simple $\mathbb{R}$-algebra $(A,\sigma)$ with involution of orthogonal type. Then $D^b_{S_n}(X^n)$ admits a full weak exceptional collection. 
\end{thm}
\begin{proof}
	We start with $\mathrm{dim}(X)=1$. Note that in this case the twisted quadric is a Brauer--Severi curve over $\mathbb{R}$. Then $S^n(X)$ is smooth for any integer $n>0$ and $S^n(X)\otimes_k k^s\simeq S^n(X\otimes_k k^s)\simeq S^n(\mathbb{P}^1)\simeq \mathbb{P}^n$. Therefore $S^n(X)$ is a Brauer--Severi variety which has a full weak exceptional collection over any field $k$ (see \cite{ORLT}, Example 1.17). We first prove the assertion for $X$ being a non-split Brauer--Severi variety. We will see that the split case can be proved in the same way.
	
	So let $X$ be a Brauer--Severi variety of dimension $2r-1\geq 2$ corresponding to the central simple algebra $M_r(\mathbb{H})$.  Note that the period of $X$ is two. Therefore $\mathcal{O}_X(2)$ exists in $\mathrm{Pic}(X)$ (see \cite{AT}). From \cite{NOT}, Section 6 we know that there is a indecomposable vector bundle $\mathcal{V}_1$ on $X$ such that $\mathcal{V}_1\otimes_k k^s\simeq \mathcal{O}_{\mathbb{P}^r}(1)^{\oplus 2}$. If $X$ is split, $\mathcal{V}_1\simeq \mathcal{O}_{\mathbb{P}^r}(1)$. This vector bundle is unique up to isomorphism and $\mathrm{End}(\mathcal{V}_1)$ is isomorphic to the central division algebra $D$ for which $M_r(D)$ corresponds to $X$. So if $X$ is split, $\mathrm{End}(\mathcal{V}_1)\simeq \mathbb{R}$. Otherwise, we have $\mathrm{End}(\mathcal{V}_1)\simeq \mathbb{H}$. Moreover, it is well-known that the collection
	\begin{eqnarray}
		\{\mathcal{O}_X,\mathcal{V}_1,\mathcal{O}_X(2), \mathcal{V}_1\otimes\mathcal{O}_X(2),...,\mathcal{O}_X(2(r-1)),\mathcal{V}_1\otimes\mathcal{O}_X(2(r-1))\}
	\end{eqnarray}   
	is a full weak exceptional collection in $D^b(X)$ (see \cite{ORLT}, Example 1.17). For simplicity, let us denote the collection (1) by  
	\begin{eqnarray*}
		\{\mathcal{V}_0,\mathcal{V}_1,\mathcal{V}_2,...,\mathcal{V}_{2(r-1)},\mathcal{V}_{2r-1}\}.
	\end{eqnarray*}
	For every multi-index $\alpha=(\alpha_1,...,\alpha_n)\in \{0,1,...,2r-1\}^n$ we write
	\begin{eqnarray*}
		\mathcal{V}(\alpha):=\mathcal{V}_{\alpha_1}\boxtimes\cdots\boxtimes \mathcal{V}_{\alpha_n}.
	\end{eqnarray*}
	This is a vector bundle on $X^n$ and it is easy to check that the collection consisting of the $\mathcal{V}(\alpha)$, ordered by the lexicographical order, generates $D^b(X^n)$. Now we want to reorder the sequence consisting of these $\mathcal{V}(\alpha)$. So for a multi-index $\alpha$, we follow \cite{KRST}, Section 4 and write for the unique non-decreasing representative of its $S_n$ orbit $\mathrm{nd}(\alpha)$. Then one can define a total order $\triangleleft$ on $\{0,1,...,2r-1\}^n$ by
	\[
	\alpha\triangleleft\beta \Longleftrightarrow \left\{\begin{array}{ll} \mathrm{nd}(\alpha)<_{\mathrm{lex}} \mathrm{nd}(\beta) & \textnormal{or}\\
	\mathrm{nd}(\alpha)=\mathrm{nd}(\beta) & \textnormal{and}\ \alpha<_{\mathrm{lex}}\beta \end{array}\right.
	\]
	where $<_{\mathrm{lex}}$ stands for the lexicographical order on $\{0,1,...,2r-1\}^n$. For details we refer to \cite{KRST}. Now the group $S_n$ acts transitively on the blocks consisting of all $\mathcal{V}(\alpha)$ with fixed $\mathrm{nd}(\alpha)$ because $\sigma^*\mathcal{V}(\alpha)\simeq \mathcal{V}(\sigma^{-1}\cdot \alpha)$. Furthermore, any $\mathcal{V}(\alpha)$ has a canonical $\mathrm{Stab}(\alpha)$-linearization given by permutation of the factors in the box product.
	If $\alpha$ is a non-decreasing multi-index and $V^{(\alpha)}_i$ an irreducible representation of $H_{\alpha}:=\mathrm{Stab}(\alpha)$, we can get a full weak exceptional collection out of the collection consisting of the vector bundles $\mathrm{Inf}^{S_n}_{H_{\alpha}}(\mathcal{V}(\alpha)\otimes V^{(\alpha)}_i)$. To get this exceptional collection, we first consider $\mathrm{End}_{S_n}(\mathrm{Inf}^{S_n}_{H_{\alpha}}(\mathcal{V}(\alpha)\otimes V^{(\alpha)}_i))$. In particular, the proof of Theorem 2.12 in \cite{ELT} and the K\"unneth-formula show that there are isomorphisms
	\begin{eqnarray*}
		\mathrm{End}_{S_n}(\mathrm{Inf}^{S_n}_{H_{\alpha}}(\mathcal{V}(\alpha)\otimes V^{(\alpha)}_i))&\simeq&\mathrm{Hom}_{H_{\alpha}}(\mathrm{Res}^{S_n}_{H_{\alpha}}\mathrm{Inf}^{S_n}_{H_{\alpha}}(\mathcal{V}(\alpha)\otimes V^{(\alpha)}_i),\mathcal{V}(\alpha)\otimes V^{(\alpha)}_i)\\
		&\simeq& \mathrm{Hom}_{H_{\alpha}}(\mathcal{V}(\alpha)\otimes V^{(\alpha)}_i,\mathcal{V}(\alpha)\otimes V^{(\alpha)}_i)\\
		&\simeq& (\mathrm{End}(\mathcal{V}_{\alpha_1})\otimes\cdots\otimes \mathrm{End}(\mathcal{V}_{\alpha_n})\otimes \mathrm{End}( V^{(\alpha)}_i))^{H_{\alpha}}.
	\end{eqnarray*}
	The case $n=1$ is clear, since the Brauer--Severi variety $X$ admits a full weak exceptional collection (see \cite{ORLT}, Example 1.17). 
	For $n>1$ we notice that the endomorphism ring of any irreducible representation of a finite group over $\mathbb{R}$ is isomorphic to $\mathbb{R}, \mathbb{C}$ or $\mathbb{H}$ according to Schur's Lemma and a theorem of Frobenius. From the construction of the collection (1) we see that $\mathrm{End}(\mathcal{V}_{\alpha_1})\otimes\cdots\otimes \mathrm{End}(\mathcal{V}_{\alpha_n})$ must be isomorphic to either $M_s(\mathbb{R})$ or $M_t(\mathbb{H})$ for suitable positive integers $s$ and $t$. Therefore, the finite-dimensional $\mathbb{R}$-algebra
	\begin{eqnarray*}
		\mathrm{End}(\mathcal{V}_{\alpha_1})\otimes\cdots\otimes \mathrm{End}(\mathcal{V}_{\alpha_n})\otimes \mathrm{End}( V^{(\alpha)}_i)
	\end{eqnarray*}
	must be isomorphic to $M_{s'}(\mathbb{R}), M_{r'}(\mathbb{C})$ or $M_{t'}(\mathbb{H})$ for suitable positive integers $s', r'$ and $t'$.
	Since $S_n$, and hence $H_{\alpha}$, acts on $\mathrm{End}(\mathcal{V}_{\alpha_1})\otimes\cdots\otimes \mathrm{End}(\mathcal{V}_{\alpha_n})\otimes \mathrm{End}( V^{(\alpha)}_i)$ by automorphism and any automorphism of a matrix-ring over a unique factorization domain is inner, \cite{MONT}, Corollary 2.13 implies that there are simple rings $A_1,...,A_{l(\alpha,i)}$ such that 
	\begin{eqnarray*}
		(\mathrm{End}(\mathcal{V}_{\alpha_1})\otimes\cdots\otimes \mathrm{End}(\mathcal{V}_{\alpha_n})\otimes \mathrm{End}( V^{(\alpha)}_i))^{H_{\alpha}}\simeq A_1\times\cdots \times A_{l(\alpha,i)}
	\end{eqnarray*}
	Clearly, the rings $A_1,...,A_{l(\alpha,i)}$ must be finite dimensional $\mathbb{R}$-algebras. Below we have to deal with positive integers $l(\alpha,i), m(\alpha,i)_j$ and $h(\alpha,i)$, depending on $\alpha$ and $i$. For a better readability, we simply write $l,m_j$ and $h$. As $\mathrm{Inf}^{S_n}_{H_{\alpha}}(\mathcal{V}(\alpha)\otimes V^{(\alpha)}_i)$ is a $S_n$-equivariant vector bundle, we apply the Krull--Schmidt Theorem to decompose it into a direct sum of indecomposables in the category of equivariant coherent sheaves $\mathrm{Coh}_{S_n}(X^n)$. Let
	\begin{eqnarray*}
		\mathrm{Inf}^{S_n}_{H_{\alpha}}(\mathcal{V}(\alpha)\otimes V^{(\alpha)}_i)=T(\alpha,i)_1^{\oplus m_1}\oplus \cdots \oplus T(\alpha,i)_{h}^{\oplus m_{h}}
	\end{eqnarray*}
	be this decomposition. We have seen above that
	\begin{eqnarray*}
		\mathrm{End}_{S_n}(T(\alpha,i)_1^{\oplus m_1}\oplus \cdots \oplus T(\alpha,i)_{h}^{\oplus m_{h}})\simeq A_1\times\cdots \times A_{l}.
	\end{eqnarray*}
	This implies $h=l$ and $\mathrm{End}_{S_n}(T(\alpha,i)_j^{\oplus m_j})\simeq A_j$. Combining a theorem of Frobenius with the Wedderburn Theorem, we obtain that $A_j$ is isomorphic to a matrix algebra over $\mathbb{R}, \mathbb{C}$ or $\mathbb{H}$. We claim that for non-decreasing $\alpha$, the collection of blocks $\{T(\alpha,i)_1,...,T(\alpha,i)_l\}$ forms a full weak exceptional collection. That the collection of blocks $\{T(\alpha,i)_1,...,T(\alpha,i)_l\}$ generates $D^b_{S_n}(X^n)$ follows from the fact that the collection consisting of the vector bundles $\mathrm{Inf}^{S_n}_{H_{\alpha}}(\mathcal{V}(\alpha)\otimes V^{(\alpha)}_i)$ generates $D^b_{S_n}(X^n)$. The argument for this fact is part of the proof of Theorem 2.12 in \cite{ELT}. For convenience of the reader, we recall the argument. Take any $\mathcal{F}\in D^b_{S_n}(X^n)$, $\mathcal{F}\neq 0$. As mentioned before, the collection consisting of the $\mathcal{V}(\alpha)$ generates $D^b(X^n)$. So for some $p$ and $\alpha$ we will have $\mathrm{Hom}^p(\mathcal{F},\mathcal{V}(\alpha))\neq 0$. Hence $\mathrm{RHom}(\mathcal{F},\mathcal{V}(\alpha))\neq 0$. Denote by $V$ the object $\mathrm{RHom}(\mathcal{F},\mathcal{V}(\alpha))^*$. Then because the functors $\mathrm{RHom}(-,\mathcal{V}(\alpha))^*$ and $F_{\mathcal{V}(\alpha)}:=\mathcal{V}(\alpha)\otimes -$ are adjoint, we find
	\begin{eqnarray*}
		\mathrm{Hom}_{S_n}(\mathcal{F},\mathrm{Inf}^{S_n}_{H_{\alpha}}(\mathcal{V}(\alpha)\otimes V))&\simeq& \mathrm{Hom}_{H_{\alpha}}(\mathcal{F},\mathcal{V}(\alpha)\otimes V)\\
		&\simeq& \mathrm{Hom}_{H_{\alpha}}(\mathrm{RHom}(\mathcal{F},\mathcal{V}(\alpha))^*,V)\\
		&\simeq& \mathrm{Hom}_{H_{\alpha}}(V,V)\neq 0.
	\end{eqnarray*}
	This proves that the collection of vector bundles $\mathrm{Inf}^{S_n}_{H_{\alpha}}(\mathcal{V}(\alpha)\otimes V^{(\alpha)}_i)$, and therefore the collection of blocks $\{T(\alpha,i)_1,...,T(\alpha,i)_l\}$, generates $D^b_{S_n}(X^n)$. So it remains to show that any $T(\alpha,i)_j$ is a weak exceptional object and the the collection of blocks $\{T(\alpha,i)_1,...,T(\alpha,i)_l\}$ forms a weak exceptional collection. Note that $\mathrm{End}_{S_n}(T(\alpha,i)_j)$ is a division algebra by construction.
	
	The proof of Theorem 2.12 in \cite{ELT} shows
	\begin{eqnarray}
		\mathrm{Ext}^d(\mathrm{Inf}^{S_n}_{H_{\alpha}}(\mathcal{V}(\alpha)\otimes V^{(\alpha)}_i),\mathrm{Inf}^{S_n}_{H_{\alpha}}(\mathcal{V}(\alpha)\otimes V^{(\alpha)}_i))=0,\ \textnormal{for}\ d>0.
	\end{eqnarray}
	Therefore $\mathrm{Ext}^d(T(\alpha,i)_j,T(\alpha,i)_j)=0$ for $d>0$. This yields that $T(\alpha,i)_j$ is a weak exceptional object. From (2) we can also conclude that within a block $\{T(\alpha,i)_1,...,T(\alpha,i)_l\}$, the following holds:
	\begin{eqnarray*}
		\mathrm{Ext}^d(T(\alpha,i)_a, T(\alpha,i)_b)=0,\ \textnormal{for}\ d>0
	\end{eqnarray*}
	if $a>b$. Moreover, the proof of Theorem 2.12 in \cite{ELT} also shows
	\begin{eqnarray*}
		\mathrm{Ext}^d(\mathrm{Inf}^{S_n}_{H_{\beta}}(\mathcal{V}(\beta)\otimes V^{(\alpha)}_j),\mathrm{Inf}^{S_n}_{H_{\alpha}}(\mathcal{V}(\alpha)\otimes V^{(\alpha)}_i))=0,\ \textnormal{for}\ d>0
	\end{eqnarray*}
	whenever $\alpha\triangleleft \beta$. But this implies
	\begin{eqnarray*}
		\mathrm{Ext}^d(T(\beta,j)_a,T(\alpha,i)_b)=0,\ \textnormal{for}\ d>0
	\end{eqnarray*}
	whenever $\alpha\triangleleft \beta$. This shows that the collection of blocks $\{T(\alpha,i)_1,...,T(\alpha,i)_l\}$ forms a full weak exceptional collection.

	If $X$ is split, i.e. isomorphic to $\mathbb{P}^m$, one can repeat the above argument with the collection
	\begin{eqnarray}
		\{\mathcal{O},\mathcal{O}(1),...,\mathcal{O}(m-1),\mathcal{O}(m)\}.
	\end{eqnarray}
	Denote this collection by $\{\mathcal{E}_0,...,\mathcal{E}_m\}$ and consider multi-indices $\alpha\in\{0,1,...,m\}^n$. Again, we write
	\begin{eqnarray*}
		\mathcal{E}(\alpha):=\mathcal{E}_{\alpha_1}\boxtimes\cdots\boxtimes \mathcal{E}_{\alpha_n}.
	\end{eqnarray*}
	Since the collection (3) is a full exceptional collection for $D^b(\mathbb{P}^m)$ and $(-)^{H_{\alpha}}$ is exact, we have
	\begin{eqnarray*}
		\mathrm{End}_{S_n}(\mathrm{Inf}^{S_n}_{H_{\alpha}}(\mathcal{E}(\alpha)\otimes V^{(\alpha)}_i))\simeq (\mathrm{End}(V_i^{(\alpha)}))^{H_{\alpha}}\simeq \mathrm{End}_{H_{\alpha}}(V_i^{(\alpha)}).
	\end{eqnarray*}
	Note that $\mathrm{End}_{H_{\alpha}}(V_i^{(\alpha)})$ is isomorphic to $\mathbb{R},\mathbb{C}$ or $\mathbb{H}$. Now repeat the above arguments to conclude that the collection of vector bundles $\mathrm{Inf}^{S_n}_{H_{\alpha}}(\mathcal{E}(\alpha)\otimes V^{(\alpha)}_i)$ forms a full weak exceptional collection for $D^b_{S_n}((\mathbb{P}^m)^n)$. Now consider the case where $X$ is a twisted quadric associated to a central simple $\mathbb{R}$-algebra $(A,\sigma)$ with involution of orthogonal type. Let $n$ be the degree of $A$. According to \cite{BLUT}, there is a semiorthogonal decomposition with exactly $n-1$ components
	\begin{eqnarray*}
		D^b({_\gamma}Q)=\langle D^b(\mathbb{R}), D^b(A),...,D^b(\mathbb{R}), D^b(A), D^b(C(A,\sigma))\rangle.
	\end{eqnarray*}
	Note that this semiorthogonal decomposition is induced by vector bundles $\mathcal{V}_1, \mathcal{V}_2,...,\mathcal{V}_{n-1}$ on ${_\gamma}Q$ satisfying $\mathrm{End}(\mathcal{V}_1)=\mathbb{R}, \mathrm{End}(\mathcal{V}_2)=A,...,\mathrm{End}(\mathcal{V}_{n-1})=C(A,\sigma)$. From \cite{KMRT}, Theorem 8.10 we know that $C(A,\sigma)$ is either a central simple algebra over $\mathbb{C}$ or that $C(A,\sigma)$ splits as the direct product of two central simple $\mathbb{R}$-algebras. In the first case $C(A,\sigma)\simeq M_{n_0}(\mathbb{C})$ whereas in the latter case $C(A,\sigma)$ is isomorphic to $A_1\times A_2$, where $A_i$ is isomorphic to $M_{n_1}(\mathbb{R})$ or $M_{n_2}(\mathbb{H})$. Now one applies the arguments for Brauer--Severi varieties from above to the full weak collection  
	collection
	\begin{eqnarray*}
		\{\mathcal{V}_1,\mathcal{V}_2,...\mathcal{V}_{n-1}\}.
	\end{eqnarray*}
	Note that in the split case, the collection $\mathcal{V}_1,...,\mathcal{V}_{n-1}$ is the full exceptional collection from Kapranov \cite{KA2T} (alternatively see \cite{BLUT}). This completes the proof.
\end{proof}
\begin{cor}
	Let $X$ be a Brauer--Severi variety over $\mathbb{R}$ or a twisted quadric associated to a central simple $\mathbb{R}$-algebra $(A,\sigma)$ with involution of orthogonal type and $1\leq n\leq 3$. We set $\mathcal{T}:=D^b_{S_n}(X^n)$. If $\mathrm{rdim}(X)=0$, then $\mathrm{rdim}(\mathcal{T})=0$.
\end{cor}
\begin{proof}
	The case $n=1$ is clear, since $S^1(X)=X$ and $X$ admits a full weak exceptional collection. For $n=2$, we see that $H_{\alpha}$ must be isomorphic to either the trivial group or to $S_2$ itself. In both cases $\mathrm{End}_{H_{\alpha}}( V^{(\alpha)}_i)\simeq \mathbb{R}$. It remains to consider $n=3$. The possible subgroups of $S_3$ are the trivial group, $S_2$, $A_3$ and $S_3$ itself. If $H_{\alpha}$ is either $S_2$ or $S_3$, \cite{PET}, Theorem 4.1 implies $\mathrm{End}_{H_{\alpha}}( V^{(\alpha)}_i)\simeq \mathbb{R}$. Furthermore, the representation theory of $A_3$ over $\mathbb{R}$ is well-known and gives that $\mathrm{End}_{A_3}( V^{(\alpha)}_i)$ is isomorphic to $\mathbb{R}$ or $\mathbb{C}$. If $\mathrm{rdim}(X)=0$, \cite{NO2T}, Proposition 5.1 and Theorem 5.6 show that $X$ must be split. In the case of Brauer--Severi varieties, $X$ is therefore isomorphic to $\mathbb{P}^m$ and we can consider the full exceptional collection (3) from above
	\begin{eqnarray*}
		\{\mathcal{O},\mathcal{O}(1),...,\mathcal{O}(m-1),\mathcal{O}(m)\}.
	\end{eqnarray*}
	Denote this collection by $\{\mathcal{E}_0,...,\mathcal{E}_m\}$ and consider multi-indices $\alpha\in\{0,1,...,m\}^n$ as in the proof of Theorem 6.5. Again, we write
	\begin{eqnarray*}
		\mathcal{E}(\alpha):=\mathcal{E}_{\alpha_1}\boxtimes\cdots\boxtimes \mathcal{E}_{\alpha_n}.
	\end{eqnarray*}
	Since the collection (3) is a full exceptional collection for $D^b(\mathbb{P}^m)$, we have
	\begin{eqnarray*}
		\mathrm{End}_{S_n}(\mathrm{Inf}^{S_n}_{H_{\alpha}}(\mathcal{E}(\alpha)\otimes V^{(\alpha)}_i))\simeq \mathrm{End}_{H_{\alpha}}(V_i^{(\alpha)})\simeq \mathbb{R}\ \textnormal{or}\ \mathbb{C}.
	\end{eqnarray*}
	Now the proof of Theorem 6.5 shows that the collection consisting of the vector bundles $\mathrm{Inf}^{S_n}_{H_{\alpha}}(\mathcal{E}(\alpha)\otimes V^{(\alpha)}_i)$ forms a full weak exceptional collection and gives therefore rise to a semiorthogonal decomposition. Since $\mathrm{rdim}(D^b(\mathbb{C}))=\mathrm{rdim}(D^b(\mathbb{R}))=0$ (see \cite{AB1T}, Proposition 6.1.6), we finally obtain $\mathrm{rdim}(\mathcal{T})=0$. For $X$ a twisted quadric associated to a central simple algebra $(A,\sigma)$ with involution of orthogonal type one can repeat the arguments from above. In this case one uses the full exceptional collection from Kapranov \cite{KA2T}.  
\end{proof}
\begin{thm}
	Let $X$ be a Brauer--Severi variety over $\mathbb{R}$ or a twisted quadric associated to a central simple $\mathbb{R}$-algebra $(A,\sigma)$ with involution of orthogonal type and $1\leq n\leq 3$. We set $\mathcal{T}:=D^b_{S_n}(X^n)$. Then the following hold:
	\begin{itemize}
		\item[(\textbf{i})] $\mathrm{rdim}(X)=0$ if and only if $\mathrm{rdim}(\mathcal{T})=0$.
		\item[(\textbf{ii})] $X(\mathbb{R})\neq \emptyset$ if and only if $\mathrm{rdim}(\mathcal{T})=0$.
	\end{itemize}
\end{thm}
\begin{proof}
	For $n=1$, (i) and (ii) is the content of \cite{NO2T}, Theorem 6.3 and Proposition 6.10. Note that $\mathcal{T}=D^b(X)$ and hence $\mathrm{rdim}(X)=\mathrm{rdim}(\mathcal{T})$. So we can consider $2\leq n\leq 3$. If $\mathrm{dim}(X)=1$, $S^n(X)\otimes_k k^s\simeq \mathbb{P}^n$ and therefore $S^n(X)$ is a Brauer--Severi variety. Again, (i) and (ii) follows from \cite{NO2T}, Theorem 6.3. Now we assume $\mathrm{dim}(X)>1$. We prove the statement only in the Brauer--Severi case and use the full weak exceptional collection (1). For $n=2$, we consider $\alpha=(1,2)$ and see $\mathrm{Stab}(\alpha)=\{\mathrm{id}\}$. Analogously, for $n=3$ we consider $\alpha=(0,1,2)$ and observe $\mathrm{Stab}(\alpha)=\{\mathrm{id}\}$. So in both of these cases we have
	\begin{eqnarray}
		\mathrm{End}_{S_n}(\mathrm{Inf}^{S_n}_{\{\mathrm{id}\}}(\mathcal{V}(\alpha)\otimes V^{(\alpha)}_i))\simeq \mathrm{End}(\mathcal{V}_1).
	\end{eqnarray}
	Recall, that $\mathrm{End}(\mathcal{V}_1)$ is isomorphic to the central division algebra $D$ for which $M_r(D)$ corresponds to $X$.
	Now we prove (i) for $\mathrm{dim}(X)>1$ and $2\leq n\leq 3$. Assume $\mathrm{rdim}(X)=0$. Then Corollary 6.6 gives $\mathrm{rdim}(\mathcal{T})=0$. On the other hand, if $\mathrm{rdim}(\mathcal{T})=0$, the derived category $\mathcal{T}=D^b_{S_n}(X^n)$ must have a semiorthogonal decomposition of the form
	\begin{eqnarray}
		D^b_{S_n}(X^n)=\langle \mathcal{A}_1,...,\mathcal{A}_e\rangle
	\end{eqnarray}
	with $\mathcal{A}_i\simeq D^b(\mathbb{R},K_i)$and $K_i$ being \'etale $\mathbb{R}$-algebras (see \cite{AB1T}, Proposition 6.1.6). We remark that $K_i\simeq \mathbb{R}^{\times n_i}\times\mathbb{C}^{\times m_i}$. Now \cite{ABT}, Lemma 1.17 implies 
	\begin{eqnarray*}
		D^b(\mathbb{R},K_i)\simeq D^b(\mathbb{R},\mathbb{R})^{\times n_i}\times D^b(\mathbb{R},\mathbb{C})^{\times m_i}. 
	\end{eqnarray*}
	Therefore we get a semiorthogonal decomposition given by
	\begin{eqnarray*}
		D^b_{S_n}(X^n)=\langle \mathcal{G}_i^{(1)},...,\mathcal{G}_i^{(n_i)},\mathcal{F}_i^{(1)},...,\mathcal{F}_i^{(m_i)}\rangle_{i=1,...,e}
	\end{eqnarray*}
	where $\mathrm{End}_{S_n}(\mathcal{G}_i^{(l)})\simeq \mathbb{R}$ and $\mathrm{End}_{S_n}(\mathcal{F}_i^{(l)})\simeq \mathbb{C}$. 
	Now Theorem 6.5 states that $\mathcal{T}$ admits a full weak exceptional collection and its proof in particular shows that the endomorphism algebras of the weak exceptional vector bundles involved are isomorphic to $\mathbb{R},\mathbb{C}$ or $\mathbb{H}$, considered as (simple) $\mathbb{R}$-algebras. Notice that one of the vector bundles occurring in the full weak exceptional collection is $\mathrm{Inf}^{S_n}_{\{\mathrm{id}\}}(\mathcal{V}(\alpha)\otimes V^{(\alpha)}_i)$ of (4). It is indecomposable, since its endomorphism algebra is a central division algebra over $\mathbb{R}$.
	
	Now let $d$ be the number of vector bundles within the set of full weak exceptional collection with endomorphism algebra being isomorphic to $\mathbb{C}$ and $r$ the number of the remaining exceptional vector bundles. We denote the full weak exceptional collection given by Theorem 6.5 simply by
	\begin{eqnarray}
		D^b_{S_n}(X^n)=\{\mathcal{E}_1,\mathcal{E}_2,...,\mathcal{E}_{r+d}\}.
	\end{eqnarray}
	The rank of the Grothendieck group $K_0(\mathcal{T})$ equals $r+d$, i.e. $K_0(\mathcal{T})\simeq \mathbb{Z}^{\oplus (r+d)}$. Note that $r+d=\sum^e_{i=1}{(n_i+m_i)}$. For a $S_n$-equivariant object $\mathcal{V}\in D^b_{S_n}(X^n)$ with $\mathrm{End}_{S_n}(\mathcal{V})\simeq \mathbb{C}$, considered as an $\mathbb{R}$-algebra, we obtain after base change $\mathrm{End}_{S_n}(\mathcal{V}\otimes_{\mathbb{R}}\mathbb{C})\simeq \mathrm{End}_{S_n}(\mathcal{V})\otimes_{\mathbb{R}} \mathbb{C}$ and hence 
	\begin{eqnarray}
		\langle \mathrm{End}_{S_n}(\mathcal{V})\otimes_{\mathbb{R}} \mathbb{C}\rangle\simeq D^b(\mathbb{C},\mathbb{C})\times D^b(\mathbb{C},\mathbb{C}).
	\end{eqnarray}
	For $\mathcal{E}\in D^b_{S_n}(X^n)$, we write $\bar{\mathcal{E}}:=\mathcal{E}\otimes_{\mathbb{R}}\mathbb{C}\in D^b_{S_n}(X^n_{\mathbb{C}})$ for the equivariant object after scalar extension. We obtain semiorthogonal decompositions
	\begin{eqnarray}
		\mathcal{T}':=D^b_{S_n}(X^n_{\mathbb{C}})=\langle \bar{\mathcal{G}}_i^{(1)},...,\bar{\mathcal{G}}_i^{(n_i)},\bar{\mathcal{F}}_i^{(1)},...,\bar{\mathcal{F}}_i^{(m_i)}\rangle_{i=1,...,e}
	\end{eqnarray}
	and 
	\begin{eqnarray}
		\mathcal{T}'=D^b_{S_n}(X^n_{\mathbb{C}})=\langle \bar{\mathcal{E}}_1,\bar{\mathcal{E}}_2,...,\bar{\mathcal{E}}_{r+d}\rangle 
	\end{eqnarray}
	Now (7) and (9) tell us that after base change to $\mathbb{C}$ the Grothendieck group $K_0(\mathcal{T}')$ has rank $r+2d$. The semiorthogonal decomposition (8) then implies $r+2d=\sum^e_{i=1}{(n_i+2m_i)}$. Since $r+d=\sum^e_{i=1}{(n_i+m_i)}$, we find $d=\sum^e_{i=1}{m_i}$. Note that $D^b_{S_n}(X^n)$ admits a unique dg-enhancement, denoted by $_{dg}D^b_{S_n}(X^n)$. In fact this follows from the results in \cite{ET} and the well-known fact that $D^b(X)$ admits a unique dg-enhancement. Alternatively see \cite{BLST}, since $[X^n/S_n]$ is a Deligne--Mumford stack and $D^b([X^n/S_n])\simeq D^b_{S_n}(X^n)$. As explained in (4), there is at least one bundle within the full weak exceptional collection in $D^b_{S_n}(X^n)$ whose endomorphism algebra is isomorphic to the central simple $\mathbb{R}$-algebra corresponding to $X$. The above semiorthogonal decompositions (5) and (6) show that the noncommutative motive $U(_{dg}D^b_{S_n}(X^n))$ decomposes as
	\begin{eqnarray}
		\bigoplus^r_{j=1}U(A_j)\oplus U(\mathbb{C})^{\oplus d}\simeq U(_{dg}D^b_{S_n}(X^n))\simeq \bigoplus^e_{i=1}\left(U(\mathbb{R})^{\oplus n_i}\oplus U(\mathbb{C})^{\oplus m_i}\right),
	\end{eqnarray}
	with $A_j$ being central simple $\mathbb{R}$-algebras. As mentioned before, there exists a $j_0\in\{1,...,r\}$ such that $A_{j_0}$ is the central simple algebra corresponding to $X$. Since $d=\sum^e_{i=1}{m_i}$, Proposition 5.2 implies
	\begin{eqnarray*}
		\bigoplus^r_{j=1}U(A_j)\simeq \bigoplus^e_{i=1}U(\mathbb{R})^{\oplus n_i}.
	\end{eqnarray*}
	Then Theorem 5.1 yields that $X$ is split, i.e. $X\simeq \mathbb{P}^{\mathrm{dim}(X)}$. From the well-known fact that the projective space admits a full exceptional collection we conclude $\mathrm{rdim}(X)=0$.
	
	Now we prove (ii). 
	If $X(\mathbb{R})\neq\emptyset$, \cite{NO2T}, Theorem 6.3 implies $\mathrm{rdim}(X)=0$. Now (i) gives $\mathrm{rdim}(\mathcal{T})=0$. Now, if $\mathrm{rdim}(\mathcal{T})=0$, we conclude from (i) and \cite{NO2T}, Proposition 5.1 that $X$ admits a full exceptional collection. But then \cite{NO2T}, Theorem 6.3 implies $X(\mathbb{R})\neq \emptyset$. In the case $X$ is a twisted quadric associated to a central simple algebra $(A,\sigma)$ with involution of orthogonal type one can use the full weak exceptional collection from \cite{BLUT}. Then repeat the arguments from above to conclude that $\mathrm{rdim}(\mathcal{T})$ implies $\mathrm{rdim}(X)=0$. The other implication is the content of Corollary 6.6. This shows (i). To prove (ii) one can use \cite{NO2T}, Theorem 5.6 and Proposition 6.10.
\end{proof}
\begin{rema}
	\textnormal{The proof of Theorem 6.7 in particular shows that the implication $\mathrm{rdim}(\mathcal{T})=0 \Rightarrow \mathrm{rdim}(X)=0$ holds for arbitrary positive integers $n$. We believe that the other implication does not hold for arbitrary $n$. The proofs of Theorem 6.5, Corollary 6.6 and Theorem 6.7 need $	\mathrm{End}(\mathcal{V}_{\alpha_1})\otimes\cdots\otimes \mathrm{End}(\mathcal{V}_{\alpha_n})\otimes \mathrm{End}( V^{(\alpha)}_i)$ to be isomorphic to matrix algebras over $\mathbb{R}, \mathbb{C}$ or $\mathbb{H}$. To make the proof work over arbitrary fields, we must ensure that any automorophism of $	\mathrm{End}(\mathcal{V}_{\alpha_1})\otimes\cdots\otimes \mathrm{End}(\mathcal{V}_{\alpha_n})\otimes \mathrm{End}( V^{(\alpha)}_i)$ is inner. We do not know if this is indeed true. We also do not know whether Theorem 6.5 for instance holds for twisted flags of type $A_n$.}
\end{rema}
It is worth to mention that if $X$ is split (and admits a full exceptional collection), $D^b_{S_3}(X^3)$ cannot have a full exceptional collection. Indeed, as mentioned in the proof of Corollary 6.6 there is at least one bundle $\mathrm{Inf}^{S_3}_{H_{\alpha}}(\mathcal{E}(\alpha)\otimes V^{(\alpha)}_i)$ within the full weak exceptional collection such that $\mathrm{End}_{S_3}(\mathrm{Inf}^{S_n}_{H_{\alpha}}(\mathcal{E}(\alpha)\otimes V^{(\alpha)}_i))\simeq \mathbb{C}$. The existence of a full exceptional collection in $D^b_{S_3}(X^3)$ would give a decomposition of the noncommutative motive as
\begin{eqnarray*}
	\bigoplus^r_{j=1}U(\mathbb{R})\oplus U(\mathbb{C})^{\oplus d}\simeq U(_{dg}D^b_{S_3}(X^3))\simeq \bigoplus^e_{i=1}U(\mathbb{R}),
\end{eqnarray*}
which is impossible. This follows from \cite{TAT}, Corollary 2.13.
\begin{cor}
	Let $X$ be a Brauer--Severi variety over $\mathbb{R}$ corresponding to $A$ and assume $\mathrm{deg}(A)>3$. Then $S^3(X)(\mathbb{R})\neq \emptyset$ if and only if $\mathrm{rdim}(\mathcal{T})=0$.
\end{cor}
\begin{proof}
	Let $X_n$ be the generalized Brauer--Severi variety associated to $A$. From \cite{KST}, Theorem 1.5 we know that $S^3(X)$ is birational to $X_3\times \mathbb{P}^{6}$. Now $X_n$ admits a $\mathbb{R}$-rational point if and only if $\mathrm{ind}(A)$ divides $n$ (see \cite{KMRT}, Proposition 1.17). Our assumption $S^3(X)(\mathbb{R})\neq \emptyset$ therefore implies $\mathrm{ind}(A)=1$ and hence $X_3$ must be a Grassmannian over $\mathbb{R}$. One can show that this implies $X(\mathbb{R})\neq \emptyset$ and hence $A$ must be split. But then $\mathrm{rdim}(X)=0$ and Theorem 6.7 gives $\mathrm{rdim}(\mathcal{T})=0$. On the other hand, $\mathrm{rdim}(\mathcal{T})=0$ implies $\mathrm{rdim}(X)=0$. From \cite{NO2T}, Proposition 5.1 we conclude that $X$, and so $A$, must be split. Therefore $X_3$ is a Grassmannian and the Lang--Nishimura Theorem provides us with a $\mathbb{R}$-rational point on $S^3(X)$.  
\end{proof}
\begin{prop}
	Let $C$ be a non-split Brauer--Severi curve over the field $\mathbb{R}$ and $2\leq n\leq 3$. We set $\mathcal{T}:=D^b_{S_n}(C^n)$. Then $\mathrm{rdim}(\mathcal{T})=1$.
\end{prop}
\begin{proof}
	This follows from Theorem 6.5 as the endomorphism algebras of the weak exceptional vector bundles involved are isomorphic to either $\mathbb{R}, \mathbb{C}$ or $\mathbb{H}$. Since $\mathrm{rdim}(D^b(\mathbb{C}))=\mathrm{rdim}(D^b(\mathbb{R}))=0$ according to \cite{AB1T}, Proposition 6.1.6 and $\mathrm{rdim}(D^b(\mathbb{H}))=1$ according to \cite{AB1T}, Proposition 6.1.10, we obtain $\mathrm{rdim}(\mathcal{T})\leq1$. Notice that $\mathrm{rdim}(\mathcal{T})=0$ would imply $\mathrm{rdim}(C)=0$ by Theorem 6.7 which is impossible for $C$ being non-split.   
\end{proof}




\begin{rema}
	\textnormal{The implication in Proposition 6.10 is not an equivalence. Take for instance a non-split Brauer--Severi curve $C$ over $\mathbb{R}$. Then $S^2(C)$ is birational to $\mathbb{P}^2$. Therefore, $S^2(C)(\mathbb{R})\neq \emptyset$, whereas $\mathrm{rdim}(D^b_{S_2}(C^2))=1$.}
\end{rema}
\noindent
We recall Orlov's definition of a noncommutative $k$-scheme \cite{ORLT}.
\begin{defi}
	\textnormal{A \emph{noncommutative scheme} over $k$ is a pretriangulated dg category $\mathcal{A}$ of the form $\mathrm{perf}(E)$ for some cohomologically bounded dg $k$-algebra $E$. A noncommutative $k$-scheme $\mathcal{A}$ is called \emph{geometric} if there is a smooth and projective $k$-scheme $X$ such that $\mathcal{A}$ is an admissible subcategory of $\mathrm{perf}(X)$.}  
\end{defi}
\noindent
For noncommutative schemes one can define \emph{noncommutative resolution of singularities}. We recall the definition given in \cite{MBTT}.
\begin{defi}
	\textnormal{Let $\mathcal{A}$ be a geometric noncommutative $k$-scheme. A \emph{noncommutative resolution} of $\mathcal{A}$ is a smooth noncommutative $k$-scheme $\mathcal{B}$ with a functor $\Phi\colon \mathcal{A}\rightarrow \mathcal{B}$ inducing a fully faithful functor $H^0(\mathcal{A})\rightarrow H^0(\mathcal{B})$. If $\mathcal{A}=\mathrm{perf}(X)$ for some scheme $X$, we say that $\mathcal{B}$ is a noncommutative resolution of $X$.}
\end{defi}
\noindent
There is also another definition, called categorical resolution of singularities, which slightly differs from the definition given above and which was introduced by Lunts and Kuznetsov \cite{LKT}. Recall the following facts on Drinfeld and Verdier quotients. If $\mathcal{A}$ is an abelian category, we write $C_{ac}(\mathcal{A})$ for the full subcategory of $C(\mathcal{A})$ consisting of acyclic objects. The Verdier quotient $D(\mathcal{A})=[C(\mathcal{A})]/[C_{ac}(\mathcal{A})]$ is the derived category of $\mathcal{A}$. According to \cite{LKT}, Theorem 3.8, one has $[C(\mathcal{A})]/[C_{ac}(\mathcal{A})]\simeq [C(\mathcal{A})/C_{ac}(\mathcal{A})]$ and hence $D(\mathcal{A})\simeq[C(\mathcal{A})/C_{ac}(\mathcal{A})]$. For a separated noetherian scheme $X$ over an arbitrary field $k$, let $\mathrm{Qcoh}(X)$ be the abelian category of quasi-coherent sheaves. Denote by $\mathrm{com}(X)$ the category of unbounded complexes and by $\mathrm{com}_{ac}(X)$ the subcategory of all acyclic complexes. Note that there are enough $h\text{-}flat$ complexes in $\mathrm{com}(X)$ for any separated quasi-compact scheme. Hence the quotient $h\text{-}flat\text{-}\mathrm{com}(X)/h\text{-}flat\text{-}\mathrm{com}_{ac}(X)$ is a dg enhancement of $D(\mathrm{Qcoh}(X))$ (see \cite{ORLT}). Let us denote by $\mathcal{P}erf \text{-}X$ the full dg subcategory of $h\text{-}flat\text{-}\mathrm{com}(X)/h\text{-}flat\text{-}\mathrm{com}_{ac}(X)$ consisting of all perfect complexes. The next theorem shows that if $G$ is a finite group of automorphisms acting on a smooth projective $k$-variety $X$, the dg enhancement $\mathcal{P}erf_G\text{-}X$ is a noncommutative resolution of the quotient variety $Y:=X/G$. That $D^b_G(X)$ is a categorical resolution of singularities in the sense of \cite{LKT} is proved in \cite{RAT}.
\begin{thm}
	Let $X$ be a smooth projective $k$-variety and $G\subset \mathrm{Aut}(X)$ a finite group acting on $X$. Then $\mathcal{P}erf_G\text{-}X$ is a noncommutative resolution of the quotient variety $Y=X/G$.  
\end{thm}
\begin{proof}
	Recall that for any morphism $f\colon X\rightarrow Y$ the componentwise pullback $f^*$ preserves $h\text{-}flat$ complexes and $h\text{-}flat$ acyclic complexes (see \cite{LKT}, Lemma 3.10). Thus we obtain a dg-functor 
	\begin{eqnarray*}
		f^*\colon h\text{-}flat\text{-}\mathrm{com}(Y)/h\text{-}flat\text{-}\mathrm{com}_{ac}(Y)\longrightarrow h\text{-}flat\text{-}\mathrm{com}(X)/h\text{-}flat\text{-}\mathrm{com}_{ac}(X), 
	\end{eqnarray*}
	inducing the derived pullback $\mathbb{L}f^*$ between the derived categories. The dg-functor also induces a dg-functor 
	\begin{eqnarray*}
		f^*\colon \mathcal{P}erf \text{-}Y\longrightarrow \mathcal{P}erf\text{-}X.
	\end{eqnarray*}
	For $X$ a smooth projective $k$-variety, we have $[\mathcal{P}erf\text{-}X]\simeq D^b(X)$. Now let $G$ be a finite group acting on $X$. For $\mathcal{A}=\mathrm{Qcoh}_G(X)$ we denote by $\mathcal{P}erf_G\text{-}X$ the subcategory of $C(\mathcal{A})/C_{ac}(\mathcal{A})$ consisting of equivariant perfect complexes. It is easy to see that $[\mathcal{P}erf_G\text{-}X]\simeq D^b_G(X)$. We see that $\mathcal{P}erf_G\text{-}X$ is a dg enhancement of $D^b_G(X)$. Note that $_{dg}D^b_{G}(X)\simeq \mathcal{P}erf_G\text{-}X$. 	
	The quotient map $\pi\colon X\rightarrow X/G$ induces a dg-functor
	\begin{eqnarray*}
		\pi^*\colon \mathcal{P}erf\text{-}X/G\longrightarrow \mathcal{P}erf_G\text{-}X.
	\end{eqnarray*}
	Therefore, we obtain the derived functor 
	\begin{eqnarray*}
		\mathbb{L}\pi^*\colon \mathrm{perf}(X/G)\longrightarrow D^b_G(X)=\mathrm{perf}_G(X).
	\end{eqnarray*}
	
	From \cite{SCHT}, we conclude that $\mathcal{P}erf_G\text{-}X$ is smooth and geometric. Denote by $\pi\colon X\rightarrow X/G$ the quotient map. Then there is a functor
	\begin{eqnarray*}
		\pi_*^G\colon D_G(X)\longrightarrow D(Y),
	\end{eqnarray*}
	where $\mathcal{F}\mapsto \mathcal{F}^G$ is the functor of $G$-invariants. The left adjoint to $\pi_*^G$ is $\mathbb{L}\pi ^*$, where the $G$-action on the object $\mathbb{L}\pi ^*\mathcal{F}$ in $D(X)$ is the trivial one. From \cite{PEST}, we have $\pi_*^G\mathcal{O}_X=\mathcal{O}_Y$ and projection formula implies 
	\begin{eqnarray*}
		\pi_*^G\mathbb{L}\pi^*\simeq \mathrm{id}.
	\end{eqnarray*}
	Again by \cite{PEST}, $\omega_X=\pi^*\omega_Y$. Thus by Grothendieck duality, we find
	\begin{eqnarray*}
		\mathrm{Hom}_{D^b(Y)}(\pi_*\mathcal{F},\mathcal{G})=\mathrm{Hom}_{D^b(X)}(\mathcal{F},\mathbb{L}\pi^*\mathcal{G}),
	\end{eqnarray*}
	for all $\mathcal{F}\in D^b(X)$ and all $\mathcal{G}\in\mathrm{perf}(Y)$. As a consequence, for any $\mathcal{F}\in D^b(X)$, we have 
	\begin{eqnarray*}
		\pi_*\mathbb{R}\mathcal{H}om(\mathcal{F},\mathcal{O}_X)=\mathbb{R}\mathcal{H}om(\pi_*\mathcal{F},\mathcal{O}_Y).
	\end{eqnarray*}
	If $\mathcal{F}$ is assumed to be $G$-equivariant, we furthermore have 
	\begin{eqnarray*}
		\mathbb{R}\mathcal{H}om(\pi_*\mathcal{F},\mathcal{O}_Y)^G=\mathbb{R}\mathcal{H}om(\pi_*^G\mathcal{F},\mathcal{O}_Y).
	\end{eqnarray*}
	From the above isomorphisms, we finally obtain (see \cite{RAT}, p.684):
	\begin{eqnarray*}
		\mathrm{Hom}_{D^b(Y)}(\pi_*^G\mathcal{F},\mathcal{G})\simeq \mathrm{Hom}_{D^b_G(X)}(\mathcal{F},\mathbb{L}\pi^*\mathcal{G}).
	\end{eqnarray*}
	Hence $\mathbb{L}\pi^*$ is right adjoint to $\pi_*^G$ and well defined on $\mathrm{perf}(Y)$. And since $\pi$ is proper and finite and $Y$ is locally noetherian, we conclude $\pi^G_{*}D^b_G(X)\subset \mathrm{perf}(Y)$. Moreover, $\mathrm{perf}(Y)\subset D^b(Y)$ is a full subcategory. Therefore,
	\begin{eqnarray*}
		\mathrm{Hom}_{\mathrm{perf}(Y)}(\pi_*^G\mathcal{F},\mathcal{G})\simeq \mathrm{Hom}_{D^b_G(X)}(\mathcal{F},\mathbb{L}\pi^*\mathcal{G}).
	\end{eqnarray*}
	Note that this implies that $\mathbb{L}\pi_{*}$ is fully faithfull, showing that $\mathcal{P}erf_G\text{-}X$ is a noncommutative resolution of $Y$. 
\end{proof}	
\begin{cor}
	Let $X$ be a Brauer--Severi variety over $\mathbb{R}$ corresponding to $A$ and assume $\mathrm{deg}(A)>3$. If $S^3(X)$ is $\mathbb{R}$-rational, we have $\mathrm{rcodim}(S^3(X))\geq 2$.
\end{cor}

\addcontentsline{toc}{section}{References}

{\small MATHEMATISCHES INSTITUT, HEINRICH--HEINE--UNIVERSIT\"AT 40225 D\"USSELDORF, GERMANY}\\
E-mail adress: novakovic@math.uni-duesseldorf.de

\end{document}